\documentclass[12pt,oneside,a4paper]{article}
\usepackage{amsmath,amsthm, amssymb}
\usepackage{hyperref}
\usepackage{cleveref}

\usepackage[left=3.0cm,right=3.0cm,top=2.0cm,bottom=2.0cm]{geometry}
\usepackage[utf8]{inputenc}
\usepackage{stix}

\newtheorem{definition}{Definition}

\newtheorem{theorem}[definition]{Theorem}

\newtheorem{corollary}[definition]{Corollary}

\newtheorem{question}[definition]{Question}

\theoremstyle{remark}
\newtheorem{remark}[definition]{Remark}

\begin{document}
\title{Quasiconformal mappings and curvatures on metric measure spaces}
\author{Jialong Deng}
\date{}
\newcommand{\Addresses}{{
  \bigskip
  \footnotesize
  \textsc{Yau Mathematical Sciences Center, Tsinghua University, Beijing, China
}\par\nopagebreak
  \textit{E-mail address}: \texttt{jialongdeng@gmail.com}}}
\maketitle
\begin{abstract}
In an attempt to develop   higher-dimensional quasiconformal mappings on metric measure spaces with curvature conditions, i.e.                                                                                                                                                                                                                                                                                                                                                                                                                                                                                                                                                                                                                                                                                                                                                                                                                                                                                                                                                                                                                                                                                                                                                                                                                                                                                                                                                                                                                                                                                                                                                                                                                                                                                                                                                                                                                                                                                                                                                                                                                                                                                                                                                                                                                                                                                                                                                                                                                                                                                                                                                                                                                                                                                                                                                                                                                                                                                                                                                                                                                                                                          from  Ahlfors  to Alexandrov, we show  that  a non-collapsed $\mathrm{RCD}(0,n)$ space ($n\geq2$)  with Euclidean volume growth is an $n$-Loewner space and satisfies  the infinitesimal-to-global principle.
\end{abstract}

  Originating in  cartography that  represents the regions of the surface of the earth on a Euclidean piece of paper and   beginning in the works of Tissot \cite{zbMATH02708702}, \cite{MR3665710}, \cite{MR4321176}, Gr{\"o}tzsch \cite{zbMATH02576848}, \cite{MR4321185},   Lavrentieff \cite{zbMATH07227171},  \cite{MR4321187} and others \cite{MR3822909}, the study of   quasiconformal mappings on the Euclidean spaces $\mathbb{E}^n$  has  a rich history of over one hundred years, see the books \cite{MR2241787}, \cite{MR3642872} and the references therein.

  Beginning with Alexandrov's insight in the 1940s \cite{MR0029518},  the geometry of  metric (measure) spaces became an important part of  modern geometry \cite{MR1835418}, \cite{MR2459454}.   Since the metric structure plays an essential role in the theory of higher-dimensional   quasiconformal mappings, it is natural to see how this theory  behaves on the metrics with curvature conditions.  For example, according to Heinonen and Koskela, it is a fundamental fact that a quasiconformal homeomorphism  of the Euclidean space  $\mathbb{E}^n$  with $n\geq 2$ is quasisymmetric,  if it maps bounded sets to bounded sets \cite{zbMATH01230683}. We argue that the Euclidean metrics on $\mathbb{R}^n$  with $n\geq 2$ satisfy the infinitesimal-to-global principle.  What are the other metrics $\mathbb{R}^n$ satisfy the infinitesimal-to-global principle is a matter of interest. 
  
   We will show in the following that the metrics  $d$ that make $(\mathbb{R}^n,d, \mathcal{H}^n)$  to be    non-collapsed $\mathrm{RCD}(0,n)$ spaces with Euclidean volume growth  also satisfy  the principle.

\begin{theorem}
A quasiconformal homeomorphism $f$ of  a non-collapsed $\mathrm{RCD} (0,n)$ space with  Euclidean volume growth  and  $n\geq 2$ ($n\in \mathbb{N}$) is quasisymmetric, if it maps bounded sets to bounded sets.
\end{theorem}

The idea of the proof came from  Heinonen-Koskela \cite{zbMATH01230683}. That is, we need the following result to prove the above theorem.

  \begin{theorem}
 Non-collapsed $\mathrm{RCD} (0,n)$ spaces  $(X,d, \mathcal{H}^n)$ with  $n\geq 2$ ($n\in \mathbb{N}$) and   Euclidean volume growth are $n$-Loewner spaces.
\end{theorem}

The definitions and details will be given later.  As an application of the  two aforementioned theorems, we  can show the distortion volume inequality of quasisymmetric  mappings of  a non-collapsed $\mathrm{RCD}(0,n)$ space  $(X,d, \mathcal{H}^n)$ with  Euclidean volume growth.   

\begin{remark}
The note is a step  to answering the question  of what facts of the classical theory  on  $\mathbb{E}^n$  are applicable to  quasiconformal/quasiregular  mappings on metric measure spaces with curvature conditions.  More results are expected to be found in this direction.
\end{remark}

The paper is organized as follows. In Section \ref{Notions}, we introduce the notions of  quasiconformality and $\mathrm{RCD}$ spaces. Section \ref{results} is devoted to the main results. In Section \ref{remarks}, we give two applications of the theorems and   two further questions.

$\mathbf{Acknowledgment}$:  The author  appreciates the anonymous reviewer's
constructive feedback.  This note is a part of my proposal that was submitted to MathJobs for a postdoctoral position in May 2021. The funding is derived from a postdoctoral fellowship of Yau Mathematical Sciences Center, Tsinghua University. This note was dedicated to the mathematicians who were working in Ukraine during 2022.

\section{Preliminaries}\label{Notions}
In this section, we recall the definition of Alexandrov spaces, non-collapsed $\mathrm{RCD} (0,n)$ spaces and  quasiconformality for the reader's convenience.

 A metric space $(X,d)$ is said to be quasiconvex if there is a constant $C>0$ so that every pair of  points $x$ and $y$ in $X$ can be joined by a curve $\gamma$, whose length satisfies $l(\gamma)\leq Cd(x,y)$. A metric space $(X, d)$ is a length metric space if the distance between each pair of points equals the infimum of the lengths of curves joining the points.  Thus, a  locally compact and complete length space is quasiconvex.

  Throughout this section,  a metric space $(X, d)$ refers to a locally compact and  complete  length metric space. For $1\leq n\in \mathbb{N}$,  $\mathcal{H}^n$  refers to the $n$-dimensional Hausdorff measure of $(X,d)$. A metric measure space $(X,d,  \mathcal{H}^n)$ is a locally compact and complete length space with the full support of $n$-dimensional Hausdorff measure $\mathcal{H}^n$. 
 
 Recall that a model space in Riemannian geometry is a simply connected complete Riemannian surface with the constant sectional curvature $\kappa$.

A geodesic triangle $\bigtriangleup$ in $X$ with geodesic segments as its sides is said to satisfy the $\mathrm{CAT}(\kappa)$-inequality if it is slimmer than its comparison triangle in the model space. That is, if there is a comparison triangle $\bigtriangleup'$ in the model space with its sides of the same length as the sides of $\bigtriangleup$ such that the distance between the points on $\bigtriangleup$ is less than or equal to the distance between the corresponding points on $\bigtriangleup'$. A length metric $d$ on $X$ is said to be a locally $\mathrm{CAT}(\kappa)$-metric if every point in $X$ has a geodesically convex neighborhood, in which every geodesic triangle satisfies the $\mathrm{CAT}(\kappa)$-inequality.
 
 Similarly, an Alexandrov space $(X,d)$ with nonnegative curvature means that every point in $X$ has a geodesically convex neighborhood, in which every geodesic triangle is fatter than the comparison triangle in the Euclidean plane. More details about comparison geometry can be found in the book \cite{MR1835418} and the references therein.

 Alexandrov geometry is a generalized Riemannian manifold with sectional curvature bounded below (or above). Similarly, Lott-Sturm-Villani theory on  metric measure spaces is a synthetic   generalized Ricci curvature bounded below, see \cite{MR2480619},  \cite{MR2237206} and  \cite{MR2459454}. Later, the Riemannian curvature-dimensional condition was introduced  as a refinement in order to single out “Riemannian structures” from the “possibly Finslerian CD structures”  so that the study of $\mathrm{RCD}(K,N)$ spaces became active, see Ambrosio's survey  \cite{MR3966731} and the references therein. Recently, a generalized  scalar curvature bounded below  on  metric measure spaces was studied by the author in \cite{MR4210894}, \cite{MR4291609} and \cite{1}.  
 
  The curvature-dimension condition $\mathrm{CD}(K,N)$ of Lott–Sturm–Villani was defined through  optimal transport theory  and  the convexity properties of $N$-dimensional entropy  on the space of  all Borel probability measures over a metric spaces at the beginning. Thanks to the works \cite{MR3385639}, \cite{MR4044464}, \cite{MR4309491}, we can choose  an equivalently shorter version in the following way. 
 
 Let $(X,d,m)$ be a metric measure space with the full Borel measure $m$ and the parameters $K\in \mathbb{R}$ (lower bound on Ricci curvature) and $N\in (1, \infty)$ (upper bound on dimension) will be kept fixed. Define the Cheeger energy $\mathrm{Ch}: L^2(X, m)\to [0,\infty]$ by
 \begin{equation*}
 \mathrm{Ch}(f) := \inf \{\liminf_{i\to \infty} \int_X \mathrm{lip}^2f_idm  \mid  f_i\in \mathrm{Lip}_b(X,d)\cap L^2(X,m), \|f_i-f\|_{L^2}\to 0\},
 \end{equation*}
 where $\mathrm{Lip}_b(X,d)$  is the set of all bounded Lipschitz functions on $X$ and
 \begin{equation*}
 \mathrm{lip}f(x) := \lim_{r \to 0^+} \sup_{y \in B_r(x)\backslash\{x\}} \frac{|f(x)-f(y)|}{d(x,y)},
 \end{equation*}
if $x$ is not isolated, $\mathrm{lip} f(x):= 0$ otherwise.  Given $f \in L^2(X, m)$, a function $g\in L^2(X, m)$ is called a relaxed gradient if there exists a sequence $\{f_n\} \subset  \mathrm{Lip}(X, d)$ and $ \tilde{g} \in  L^2(X, m)$ so that
\begin{enumerate}
\item[(1).] $f_n \to f$ in $L^2(X, m)$ and $\mathrm{lip} f_n$ converge weakly to  $ \tilde{g}  \in  L^2(X, m)$.
\item[(2).] $g\geq \tilde{g} $ $m$-a.e..
\end{enumerate}
     A minimal relaxed gradient is a relaxed gradient that is minimal in $L^2$-norm in the family of relaxed gradients of $f$. If this family is non-empty, the minimal relaxed gradient which is denoted by $|\nabla f|$ exists and is unique $m$-a.e..

Then the Sobolev space $H^{1,2}= H^{1,2}(X,d, m)$ is defined as the finiteness domain of $\mathrm{Ch}$ in $L^2(X,m)$ and it is a Banach space equipped with the norm $\|f\|^2_{H^{1, 2}}:= \|f\|^2_{L^2}+ \mathrm{Ch}(f)$. $(X, d, m)$ is said to be \textit{infinitesimally Hilbertian } if $H^{1,2}(X,d, m)$ is a Hilbert space.  In
particular, for all $f_i \in H^{1,2}$ $ (i = 1, 2)$,
\begin{equation*}
<\nabla f_1, \nabla f_2>:= \lim_{t\to 0}\frac{|\nabla (f_1+tf_2)|^2- |\nabla f_1|^2}{2t}\in L^2(X,m).
\end{equation*}
is well defined.  Then,   by using the  infinitesimally Hilbertian condition,   $f\in H^{1, 2}$  is said to be in the domain of the Laplacian ($f \in D(\Delta)$) if there exists $\Delta f \in  L^2(X, m)$ so that
\begin{equation*}
\int_X g \Delta f dm + \int_X <\nabla g, \nabla f>dm = 0
\end{equation*}
 for any $g\in H^{1, 2}$.  The following definition  of $\mathrm{RCD}(K, N)$ spaces  comes from \cite[Definition~2.1]{MR4173928}.

\begin{definition}[$\mathrm{RCD}(K, N)$ spaces]
Let $(X,d,m)$ be a metric measure space, let $K\in \mathbb{R}$ and let $N\in (1, \infty)$.  We say $(X,d,m)$ is an $\mathrm{RCD}(K, N)$ spaces if the following hold:
\begin{enumerate}
\item[(1).] (Volume growth condition)  There exist $x \in X$ and $C > 1$ such that $m(B_r(x))\leq  Ce^{Cr^2}$
for all $r \in (0, \infty)$.
\item[(2).] (Riemannian structure)   The Sobolev space $H^{1,2}= H^{1,2}(X,d, m)$ is a Hilbert space. 
\item[(3).] (Sobolev-to-Lipschitz property) Any function $f\in H^{1,2}$  satisfying $|\nabla f|(y)\leq 1$ for $m$-a.e. $y\in X$ has $1$-Lipschitz representative.
\item[(4).]  (Bochner inequality)  For all $f \in D(\Delta)$ with $\Delta f \in   H^{1,2}$, 
\begin{equation*}
\frac{1}{2}\int_X|\nabla f|^2 \Delta \varphi dm \geq  \int_X \varphi [\frac{(\Delta f)^2}{N} +  <\nabla f, \nabla \Delta f> + K|\nabla f|^2] dm
\end{equation*}
for all    $\varphi \in D(\Delta) \cap  L^{\infty}(X, m)$  with $\varphi\geq 0$ and $\Delta \varphi \in L^{\infty}(X, m)$.

\end{enumerate}

\end{definition}

 In fact, a metric measure space $(X, d, m)$ is an $\mathrm{RCD}(K, N)$ space iff $(X, d, m)$ satisfies the $\mathrm{CD}(K, N)$ condition and is infinitesimally Hilbertian  \cite[Theorem~2.14]{2020arXiv200907956D}.  Typical examples of $\mathrm{RCD}$ spaces  are measured Gromov–Hausdorff limit spaces
of Riemannian manifolds with Ricci bounds from below and dimension bounds from
above, so-called Ricci limit spaces. $\mathrm{RCD}(0,n)$ spaces come out naturally as the metric cone of $\mathrm{RCD}(n-2,n-1)$  spaces.

 An $\mathrm{RCD}(K,N)$ space $(X, d, m)$ is\textit{ noncollapsed} if $n$ is a natural number and $m = \mathcal{H}^n$.  Noncollapsed $\mathrm{RCD}(K,n)$ spaces give a natural intrinsic generalization of noncollapsing Ricci limit spaces.   Furthermore, the class of $\mathrm{RCD}(K,n)$ spaces
strictly contains the noncollapsed Ricci limit spaces.   A metric cone (resp. a spherical suspension) over $\mathbb{RP}^2$ is an example of a noncollapsed $\mathrm{RCD}(0,2)$ (resp. $\mathrm{RCD}(1,3)$) space.   A convex body in $\mathbb{E}^n$ with boundary  cannot arise as a noncollapsed Ricci limit of manifolds without boundary. However, this is a noncollapsed $\mathrm{RCD}(0,n)$ space \cite[Theorem~1.10]{MR4226234}.

 Petrunin shows that an $n$-dimensional Alexandrov space with non-negative curvature and equipped with the  induced Hausdorff measure satisfies  $\mathrm{CD}(0,n)$ condition, see \cite{MR2869253} and \cite[Section~2]{MR4210894}.   A finite dimensional Alexandrov space with curvature bounded below is infinitesimally Hilbertian. Thus an $n$-dimensional Alexandrov space with non-negative curvature and equipped with the induced Hausdorff measure is a   non-collapsed $\mathrm{RCD}(0,n)$ space. 
 
The reason why  we  focus on the subject of noncollapsed $\mathrm{RCD}(0,n)$ spaces is because the  measures  are determined by the metrics. In the following,  we will elaborate on the metric definition of quasiconformality, which is one of the three commonly adopted definitions in the existing literature.
 
\begin{definition}[Quasiconformal]\label{quasiconformal}
 A homeomorphism $f:X\to Y$ between the metric spaces $(X,d_X)$ and $(Y, d_Y)$ is said to be $K$-quasiconformal if there is a constant $0<K<\infty$ so that 
\begin{equation*}
\limsup\limits_{r\to 0}\frac{L_f(x,r)}{l_f(x,r)}\leq K
\end{equation*}
for all $x\in X$, where
\begin{equation*}
L_f(x,r):=\sup\limits_{d_X(x,y)\leq r}d_Y(f(x),f(y))
\end{equation*}
and
\begin{equation*}
l_f(x,r):=\inf\limits_{d_X(x,y)\geq r}d_Y(f(x),f(y)).
\end{equation*}
\end{definition}

 Quasiconformal homeomorphisms arise not only in cartography and geometric functional theory, but also in some parts of dynamics system, topology, and Riemannian geometry, see the survey \cite{MR934326} and the references therein. 
 
  The theory of quasiconformal mappings plays a role in applied mathematics.  For example, combining   quasiconformal Teichm{\"u}ller theory  \cite{MR1730906} with scientific computing techniques,  computational quasiconformal geometry has various applications in engineering and medical imaging, see  \cite{MR2918023}, \cite{MR2908091} and the references therein.

\begin{definition}[Quasisymmetric]
 A homeomorphism $f:X\to Y$ between the metric spaces $(X,d_X)$ and $(Y, d_Y)$ is said to be quasisymmetric if there is a constant $H_f< \infty$ so that 
 \begin{equation*}
 H_f(x,r):=\frac{L_f(x,r)}{l_f(x,r)}\leq H_f
 \end{equation*}
for all $x\in X$ and all $r>0$.

Namely,
\begin{equation*}
d_X(x,a)\leq d_X(x,b)\quad implies \quad d_Y(f(x),f(a))\leq H_fd_Y(f(x),f(b))
\end{equation*}
  for each triple $x$, $a$, $b$ of points of $X$.
\end{definition}

A quasiconformal homeomorphism is a local and infinitesimal condition, whereas a quasisymmetric homeomorphism is a global condition that imposes a uniform requirement on the relative metric distortion of any triple of points.

\begin{definition}[$\eta$-quasisymmetric]
A homeomorphism $f:X\to Y$ between the metric spaces $(X,d_X)$ and $(Y, d_Y)$ is said to be $\eta$-quasisymmetric if there is a homeomorphism $\eta:[0,\infty)\to[0,\infty)$ so that 
\begin{equation*}
d_X(x,a)\leq td_X(x,b)\quad implies \quad d_Y(f(x),f(a))\leq \eta(t)d_Y(f(x),f(b))
\end{equation*}
for each $t>0$ and for each triple $x$, $a$, $b$ of points of $X$.
\end{definition}

Though $\eta$-quasisymmetry implies quasisymmetric, these two notions are not equivalent in general.   However, if $X$ and $Y$ are path-connected doubling metric spaces, then these two notions are  equivalent. 

$\eta$-quasisymmetric  mappings naturally arise in one variable complex dynamics, see \cite{zbMATH03830924}.     $\eta$-quasisymmetric mappings can be used in geometric group theory. For instance,     quasi-isometric mappings of Gromov hyperbolic spaces are in one-to-one correspondence with $\eta$-quasisymmetric mappings on the Gromov  boundaries, see \cite{zbMATH04190607}.

\begin{definition}[Ahlfors $Q$-regular]
 A metric space $(X,d)$ is said to be Ahlfors $Q$-regular if there is a constant $C\geq 1$ so that 
\begin{equation*}
C^{-1}R^Q\leq \mathcal{H}^Q(B(R))\leq CR^Q
\end{equation*}
for all $R$-balls $B(R)$ in  $X$ of radius $R$ less than the diameter of $X$ (the diameter may be infinite).
\end{definition}

Let $(M, g)$ be a complete open Riemannian $n$-manifold, then the asymptotic volume ratio of $(M,g)$ is defined as following: 
\begin{equation*}
\mathrm{AVR}_g:=\lim_{r\to \infty}\frac{\mathrm{Vol}_g(B_x(r))}{\omega_nr^n},
\end{equation*}
where $\omega_n$  is the volume of the Euclidean unit ball in $\mathbb{E}^n$.  The manifold $M$ is said to have Euclidean volume growth when $\mathrm{AVR}_g> 0$. The constant $\mathrm{AVR}_g$ is a global geometric invariant of $M$, i.e., it is independent of the base point. Also, when $\mathrm{AVR}_g> 0$, we have $\mathrm{Vol}_g(B_x(r))\geq \mathrm{AVR}_g\omega_nr^n$ for all $x\in M$ and for all $r>0$. 

Similarly, the asymptotic volume ratio can  be defined for a metric measure space $(X,d,  \mathcal{H}^n)$. That is 
\begin{equation*}
\mathrm{AVR}_d:=\lim_{r\to \infty}\frac{\mathcal{H}^n(B_x(r))}{\omega_nr^n}.
\end{equation*}

\section{Quasiconformality VS. Quasisymmetry on $\mathbb{R}^n$}\label{results}
In this section, we search  metrics with curvature conditions on $\mathbb{R}^n$ to satisfy the condition that  quasiconformal mappings are quasisymmetric.

A complete open nonnegatively curved Riemannian manifold may not have  Euclidean volume growth in general. For example,  the product metric of a standard spherical metric and a Euclidean metric  is a complete Riemannian metric with nonnegative sectional curvature and  without  Euclidean volume growth. However, a complete open non-negatively curved Riemannian manifold with  Euclidean volume growth is diffeomorphic to $\mathbb{R}^n$, since its asymptotic cone at infinity has its dimension strictly smaller than $n$ if the soul is not a point, while manifolds with Euclidean volume growth have asymptotic cones of dimension $n$. Non-negatively curved Alexandrov spaces with Euclidean volume growth do not have to be smooth or even topological manifolds.  For example, a metric cone over any ($n-1$)-dimensional Alexandrov space of curvature $\geq 1$ is Alexandrov of curvature $\geq 0$ and has Euclidean volume growth.

Assume $(\mathrm{M}, g)$ is a complete open Riemannian $n$-manifold with non-negative Ricci curvature and $\mathrm{AVR}_g=1$, then the Bishop-Gromov inequality implies that $\mathrm{M}$ is isometric to $\mathbb{E}^n$. It is known that if $n=3$ and $\mathrm{AVR}_g>0$, then $\mathrm{M}$ is contractible \cite{MR1217167}; if $n=4$ and $\mathrm{AVR}_g>0$, then the manifolds may have infinite topological types \cite{MR1779615}.  However, Perelman  shows that if $\mathrm{M}$ has  maximal Euclidean volume growth, i.e. there exists a small positive constant $a(n)$   such that  $\mathrm{AVR}_g\geq 1-a(n)>0$, then $\mathrm{M}$ is contractible and homeomorphic to $\mathbb{R}^n$ \cite[Theorem~2]{MR1231690}. Here $a(n)$ only depends  on the dimension $n\geq 2$ of the manifold. Furthermore, Cheeger and Colding show that $\mathrm{M}$ is indeed  $C^{1,\alpha}$-diffeomorphic to $\mathbb{R}^n$  on the same assumption of Perelman's theorem in \cite[Theorem~A.1.11]{MR1484888}.

\begin{remark}
However, there exists  Riemannian manifolds  with non-negative Ricci curvature and linear volume growth, thus it cannot be $Q$-regular for $Q>1$. Furthermore, there exists  Riemannian metrics $g$ with non-negative Ricci curvature  and  arbitrarily small $\mathrm{AVR}_g$ on $\mathbb{R}^n$.   For example, let  $n\geq 3$ and $f:[0,\infty)\to[0,1]$ be a smooth nonincreasing function such that $f(0)=1$ and 
\begin{equation*}
\lim\limits_{s\to\infty}f(s)=a\in(0,1],
\end{equation*}
then the warped product metric $g:=dr^2+F(r)^2d\theta^2$ is 
 the rotationally invariant  metric on $\mathbb{R}^n$. Here 
\begin{equation*}
F(r):=\int_0^rf(s)ds
\end{equation*}
 and $d\theta^2$ is the standard metric on the sphere $S^{n-1}$. If $x=(x_1,\theta_1)$ and $\tilde{x}=(x_2,\theta_2)$ are in $\mathbb{R}^n$, then one has $d_g(x,\tilde{x})\geq \|x_1-x_2\|$ and it implies that the metric $g$ is complete. One  can  show that the sectional curvature of $g$ is nonnegative and $\mathrm{AVR}_g=a^{n-1}$, see \cite[Example~2.4]{Balogh2022}.
\end{remark}

 A conformal deformation of the Euclidean metric  on  $\mathbb{R}^n$ with infinite volume and under $L^{n/2}$ scalar curvature bounds   has Euclidean volume growth, see \cite[Theorem~2.8]{MR4182081}.

 Inspired by Perelman's theorem, Kapovitch and Mondino show that  there exists  a small positive constant $\epsilon(n)$ such that if a  metric measure space  $(X,d,  \mathcal{H}^n)$ is a non-collapsed $\mathrm{RCD}(0,n)$ space and  $\mathrm{AVR}_d\geq 1-\epsilon(n)>0$,  then  $X$ is homeomorphic to $\mathbb{R}^n$ in \cite[Theorem~1.3]{MR4226234},  also see \cite[Proposition~7.1]{2022arXiv220710029H}.      Kapovitch-Mondino theorem does not hold for  $\mathrm{CD}(0,n)$ space in general. Notice that there exists  $\mathrm{CD}(0,n)$ spaces that are not   non-collapsed $\mathrm{RCD}(0,n)$ spaces. For instance,   since  $\mathbb{R}^n$ that is  endowed with the Lebesgue measure and the distance coming from a norm is a $\mathrm{CD}(0,n)$ space, one can  equip $\mathbb{R}^n$  with the $L^{\infty}$-norm such that it is not a non-collapsed  $\mathrm{RCD}(0,n)$ space.

\begin{definition}[Heinonen-Koskela]
A metric measure space $(X, d, \mu)$ is a $p$-Loewner space if there exists a decreasing function $\phi: (0,\infty)\to (0, \infty)$ so that 
\begin{equation*}
\mathrm{Mod}_p \Gamma (E, F)\geq \phi(\bigtriangleup(E,F))
\end{equation*}
for all disjoint compact connected subsets $E, F \subset X$. Here $\Gamma (E, F)$ denotes the collection of curves joining $E$ to $F$,  the $p$-modulus ($p\geq 1$) of  $\Gamma (E, F)$ is defined as 
\begin{equation*}
\mathrm{Mod}_p \Gamma (E, F):=\inf\int_X\varrho^pd\mu
\end{equation*}
where the infimum  takes over all nonnegative Borel functions $\varrho: X \to [0,\infty]$ satisfying 
\begin{equation*}
\int_\gamma\varrho ds\geq 1
\end{equation*}
for all locally rectifiable curves $\gamma\in \Gamma(E,F)$. Note that by definition the modulus of all curves in $X$ that are not locally rectifiable is zero.

And here 
\begin{equation*}
\bigtriangleup(E,F):=\frac{d(E, F)}{\min\{\mathrm{diam}(E), \mathrm{diam}(F)\}}
\end{equation*}
denotes the relative distance between E and F and $\mathrm{diam}(E)$ denotes the diameter of $E$. 

\end{definition}

We  recap the two main theorems mentioned at the beginning of the paper and  recall the definition of non-branching geodesic space   for   convenience.  Let $\mathrm{Geo}(X)$ be the space of constant speed geodesics  in the geodesic space $(X, d)$, i.e.,
\begin{equation*}
\mathrm{Geo}(X):=\{\gamma\in C([0,1]; X):  d(\gamma(t), \gamma(s))=|t-s| d(\gamma(0), \gamma(1)), \mbox{ for  any  t,  s $\in$ [0, 1]}\}.
\end{equation*}
 The geodesic space $(X, d)$ is  \textit{non-branching} iff for any $\gamma_1$, $\gamma_2\in \mathrm{Geo}(X)$ we have: if there exists $t\in (0, 1)$ such that $\gamma_1(s)=\gamma_2(s)$ for all $s\in [0, t]$, then $\gamma_1=\gamma_2$. 

\begin{theorem}\label{Loewner}
 Non-collapsed $\mathrm{RCD}(0,n)$ spaces  $(X,d, \mathcal{H}^n)$ with  $n\geq 2$ ($n\in \mathbb{N}$) and  Euclidean volume growth are $n$-Loewner spaces.
\end{theorem}

\begin{proof}
 Since  an $\mathrm{RCD}(0,n)$ space with  $n\geq 2$  is non-branching \cite[Theorem~1.3]{2020arXiv200907956D},  non-collapsed $\mathrm{RCD}(0,n)$ spaces  with  $n\geq 2$ satisfy $(1,1)$-Poincare inequality \cite[Theorem~30.26]{MR2459454}.  In addition,  non-collapsed $\mathrm{RCD}(0,n)$ spaces with Euclidean volume growth are proper, doubling and  quasiconvex,  then Heinonen-Koskela theorem implies that these  are $n$-Loewner spaces \cite[Theorem~5.7]{zbMATH01230683}.
\end{proof}

\begin{remark}
The non-collapsed condition is not necessary for this theorem to hold. Having said that,  we state the theorem with non-collapsed condition because  our target is  the  metrics   on $\mathbb{R}^n$.

\end{remark}

\begin{theorem}\label{quasisymmetry}
A quasicomformal homeomorphism $f$ of  a non-collapsed $\mathrm{RCD}(0,n)$ space with Euclidean volume growth and  $n\geq 2$ ($n\in \mathbb{N}$) is quasisymmetric, if it maps bounded sets to bounded sets.
\end{theorem}

\begin{proof}
 Since a  non-collapsed $\mathrm{RCD}(0,n)$ space $(X,d, \mathcal{H}^n)$ satisfies Bishop-Gromov inequality,  $(X,d, \mathcal{H}^n)$ with Euclidean volume growth  is  Ahlfors $n$-regular. On the other hand, Theorem \ref{Loewner} implies that it is  an $n$-Loewner space.  Then  the quasiconformal homeomorphisms  that maps bounded sets to bounded sets  between $n$-regular   Loewner spaces ($n\geq 2 $)  are quasisymmetric \cite[Corollary~4.8]{zbMATH01230683}.
\end{proof}

Thus,  the homeomorphism $f$ is also $\eta$-quasisymmetric.  We do not need any of the priori regularity assumptions on the metrics or homeomorphisms to obtain global bounds. Since the classic Liouville theorem shows that the unit balls in $\mathbb{R}^n$ with Euclidean metric  ($n\geq 3$) can be conformally equivalent to half-spaces,  the condition that maps bounded sets to bounded sets in Theorem \ref{quasisymmetry} is necessary. 

Since non-collapsed $\mathrm{RCD}(K,n)$ spaces ($K >0$) are compact according to the generalized Bonnet-Meyer theorem, a quasiconformal  homeomorphism of a  non-collapsed $\mathrm{RCD}(K,n)$ space  $(X,d, \mathcal{H}^n)$   is quasisymmetric. The proof of Theorem \ref{quasisymmetry} shows that a quasiconformal homeomorphism from non-collapsed $\mathrm{RCD}(0,n)$ spaces  $(X,d, \mathcal{H}^n)$  with maximal Euclidean volume growth and $n\geq 2$ to $\mathbb{E}^n$ is quasisymmetric.

\begin{remark}
A  separable metric space is said to be purely $n$-dimensional if every non-empty open subset has the topological dimension $n$. Then a purely $n$-dimensional, proper, geodesically complete $\mathrm{CAT}$(0) space $(X,d)$ with $\mathrm{AVR}_d(X)< 3\omega_n/2$ is  homeomorphic to $\mathbb{R}^n$ \cite[Theorem~1.2]{zbMATH07491584}.  It is not clear to the author  whether quasiconformal homeomorphism of the space satisfies the infinitesimal-to-global principle. 

\end{remark}

\begin{remark}
The hyperbolic $n$-plane $\mathbb{H}^n$ is an $n$-Loewner space. However,  the volume of balls in  $\mathbb{H}^n$ increases exponentially with respect to the radius of the ball rather than polynomially as in the Euclidean space such that $\mathbb{H}^n$  is not $n$-regular. 
\end{remark}

\begin{remark}
A  complete open Riemannian $n$-manifold with uniformly positive scalar curvature is also not $n$-regular in general. Because $\mathbb{R}^3$ can be given a complete Riemannian metric $g_s$ with scalar curvature greater than  $2$,   one can take the metric product of $\mathbb{H}^n$ and $(\mathbb{R}^3, g_s)$.  Then the scalar curvature of the product manifold is bounded below by $1$ and the volume of the balls in the product manifold has  exponential growth.
\end{remark}

\section{Final Remarks}\label{remarks}

Since the  second-order differential calculus is developed on  $\mathrm{RCD}$ spaces \cite{MR3756920}, \cite{MR4321459}, one can also develop a distortion theory of quasiconformal  mappings in a non-collapsed $\mathrm{RCD}(0,n)$ space  $(X,d, \mathcal{H}^n)$ with Euclidean volume growth,   which deals with the estimates for the modulus of continuity and change of distances under these mappings. We will give two generalized results in this section.

Assume that $f$ is an $\eta$-quasisymmetric  homeomorphism of a non-collapsed $\mathrm{RCD}(0,n)$ space  $(X,d, \mathcal{H}^n)$ with Euclidean volume growth and  $n\geq 2$, then one can define the volume derivative of $f$ at $x\in X$ as
\begin{equation*}
\mu_f(x):=\lim\limits_{r\to 0}\frac{\mathcal{H}^n(f(B_x(r)))}{\mathcal{H}^n(B_x(r))}.
\end{equation*} 
This limit exists according to the Lebesgue-Radon-Nikodym theorem and it is finite for almost every $x$ in $X$. The function $\mu_f$ is an $\mathcal{H}^n$-measurable function on $X$ and is known as the generalized Jacobian of $f$.

 Recall that a doubling Borel measure $\rho$ is $\mathrm{A}_\infty$-related to    the Hausdorff measure $\mathcal{H}^n$ in $X$ if for each $\epsilon>0$ there is $\delta>0$ such that $\mathcal{H}^n(E)<\delta\mathcal{H}^n(E) $ implies $\rho(E)<\epsilon \rho(E)$.  The function $\mu_f$ can be related to the pull-back measure $f^*\mathcal{H}^n(E):=\mathcal{H}^n(f(E))$ in the following way.

\begin{theorem}\label{A-related}
Let $f$ be an $\eta$-quasisymmetric  homeomorphism of a non-collapsed $\mathrm{RCD}(0,n)$ space  $(X,d, \mathcal{H}^n)$ with Euclidean volume growth and  $n\geq 2$, then the pull-back measure $f^*\mathcal{H}^n$ is $\mathrm{A}_\infty$-related to  $\mathcal{H}^n$ in $X$. Moreover, $\mathrm{d}f^*\mathcal{H}^n=\mu_f\mathrm{d}\mathcal{H}^n$ with $\mu_f>0$ for $\mathcal{H}^n$-almost every $x$ in $X$, and there is $\epsilon>0$ such that 
\begin{equation*}
(\intbar_B\mu_f^{1+\epsilon}\mathrm{d}\mathcal{H}^n)^{\frac{1}{1+\epsilon}}\leq C \intbar_B\mu_f\mathrm{d}\mathcal{H}^n
\end{equation*}
 for all balls $B$ in $X$, quantitatively.
\end{theorem}

\begin{proof}
 Since non-collapsed $\mathrm{RCD}(0,n)$ spaces  with  $n\geq 2$ satisfy $(1,1)$-Poincare inequality, then the pull-back measure $f^*\mathcal{H}^n(E)$ is $\mathrm{A}_\infty$-related to the Hausdorff measure $\mathcal{H}^n$ in $X$ \cite[Theorem~7.11]{zbMATH01230683}, \cite[Theorem~1.0.4]{MR2415381}. 
\end{proof}

Therefore,  Theorem \ref{A-related} implies that   $\mathcal{H}^n(E)=0$ if and only if $\mathcal{H}^n(f(E))=0$ for an $ \mathcal{H}^n$-measurable subset $E\subset X$. That is,  $f$ and its inverse are absolutely continuous.   Theorem \ref{A-related} also  implies  that  the  $\eta$-quasisymmetric  homeomorphism  $f$ preserves the dimensions of the sets of Hausdorff dimension $n$.  Furthermore, we can bound the measure of the image of a set by the measure of the set in the following way.

\begin{corollary}\label{Schwarz}
Let $f$ be  an  $\eta$-quasisymmetric  homeomorphism of a non-collapsed $\mathrm{RCD}(0,n)$ space  $(X,d, \mathcal{H}^n)$ with Euclidean volume growth and  $n\geq 2$, let $\epsilon$ be the constant in Theorem \ref{A-related} and let  $F$ be a compact subset of  $X$,   then for each $a\in (0,\epsilon/1+\epsilon)$ there exists a constant $b$ such that 
\begin{equation*}
\mathcal{H}^n(f(E))\leq b\mathcal{H}^n(E)^a
\end{equation*}
for each $\mathcal{H}^n$-measurable subset $E$ of  $F$.
\end{corollary}

\begin{proof}
Fix $a\in (0, \frac{\epsilon}{1+\epsilon})$ and let $q=\frac{1}{1-a}\in (1, 1+\epsilon)$. Since $\mu_f$ is locally $L^q$-integrable  in $X$ by Theorem \ref{A-related}, then 
\begin{equation*}
b:=(\int_E\mu_f^q\mathrm{d}\mathcal{H}^n)^{\frac{1}{q}}<\infty,
\end{equation*}
and for each measurable $E\subset F$, we have

\begin{equation*}
\mathcal{H}^n(f(E))=\int_E\mu_f\mathrm{d}\mathcal{H}^n\leq (\int_F\mu_f^q\mathrm{d}\mathcal{H}^n)^{\frac{1}{q}}\mathcal{H}^n(E)^a=b\mathcal{H}^n(E)^a.
\end{equation*}
\end{proof}

 However,  based on the existence of the Cantor sets on  $\mathbb{E}^n$,  Gehring and  V{\"a}is{\"a}l{\"a} \cite{MR324028} show that quasiconformal homeomorphisms of $\mathbb{E}^n$, $n\geq 2$, can distort the Hausdorff dimensions of the subsets, whose  Hausdorff dimensions are not zero or $n$. Thus, the quasiconformal homeomorphisms of  $\mathbb{E}^n$ can distort the perimeters of the subsets. 
 
 It is not clear to the author  whether or not quasiconformal homeomorphisms  of a non-collapsed $\mathrm{RCD}(0,n)$ space  $(X,d, \mathcal{H}^n)$ with Euclidean volume growth and  $n\geq 2$ can distort the Hausdorff dimension of subsets. Recall that   the perimeter of $E$ on the metric measure space $(X,d, \mathcal{H}^n)$ can be defined as the relaxed Minkowski content of $E$ for a finite  $ \mathcal{H}^n$-measureable set $E$  \cite[Theorem~3.6]{MR3614662}.

 Some of the examples of the distortion theory of quasiconformal  mappings in $\mathbb{E}^n$ are  the quasiconformal counterparts of the   Schwarz Lemma by Gr{\"o}tzsch \cite{zbMATH02553182},  the classical Schwarz-Pick-Ahlfors Lemma \cite{MR1501949}, \cite{MR1704258}  and Mori-Fehlmann-Vuorinen theorem \cite{MR79091}, \cite{MR975570}.

\begin{theorem}[Mori-Fehlmann-Vuorinen Theorem]
Let $\mathbb{B}^n$ be the unit ball of the Euclidean space $(\mathbb{E}^n, d_{\mathbb{E}^n})$, $n\geq 2$, and $f$ be a $K$-quasiconformal mapping of $\mathbb{B}^n$ onto $\mathbb{B}^n$ with $f(0)=0$. Then, 
\begin{equation*}
d_{\mathbb{E}^n}(f(x),f(y))\leq M(n,K)d_{\mathbb{E}^n}(x,y)^{K^\frac{1}{1-n}}
\end{equation*}
for all $x,y\in \mathbb{B}^n$ and  the constant $M(n,K)$ has the following three properties:
\begin{enumerate}
\item[(1)]  $M(n,K)\to 1$ as $K\to 1$, uniformly in $n$;
\item[(2)]  $M(n,K)$  remains bounded for fixed $K$ and varying $n$;
\item[(3)]  $M(n,K)$  remains bounded for fixed $n$ and varying $K$.
\end{enumerate}

\end{theorem}

Motivated by the similarity of the inequalities in  Corollary \ref{Schwarz} and Mori-Fehlmann-Vuorinen Theorem, one could ask the following question:
\begin{question}
Do the classical Schwarz-Pick-Ahlfors Lemma  and Mori-Fehlmann-Vuorinen Theorem hold for $\mathrm{CAT}(-1)$ spaces?
\end{question}

If we do not require the homeomorphisms mapping in  Definition \ref{quasiconformal} (of quasiconformality), then we get the definition of quasiregular mappings on metric spaces.  To answer Zorich's question, Rickman shows that   a nonconstant  quasiregular mapping $f:\mathbb{E}^n \to\mathbb{E}^n$ can only  omit finite values for $n\geq 3$ \cite{MR583633}.   Zorich \cite{zbMATH03287913} shows that any locally injective  quasiconformal mapping  $f: \mathbb{E}^n\to \mathbb{E}^n$ for $n\geq 3$ is globally injective.  
\begin{question}
Can those two classic  theorems   be  extended to non-collapsed $\mathrm{RCD}(0,n)$ spaces  $(X,d, \mathcal{H}^n)$ with Euclidean volume growth and $n\geq 3$?
\end{question}

\bibliographystyle{alpha}
\bibliography{reference}
\Addresses

\end{document}